\documentclass[11pt]{amsart}
\usepackage{graphicx}
\usepackage{color, fullpage, amsmath, amsthm, amssymb,verbatim,enumerate}

\newtheorem{theorem}{Theorem}
\newtheorem{lemma}[theorem]{Lemma}
\newtheorem{proposition}[theorem]{Proposition}
\newtheorem{corollary}[theorem]{Corollary}

\theoremstyle{definition}
\newtheorem*{rem}{Remark}

\theoremstyle{definition}

\theoremstyle{definition}

\newcommand{\Z}{\mathbb{Z}}
\newcommand{\R}{\mathbb{R}}
\newcommand{\N}{\mathbb{N}}

\newcommand\interior{\operatorname{int}}

\begin{document}

\date{\today}
\title{Interpolating quasiregular power mappings}
\author{Jack Burkart}
\email{jburkart@simons-rock.edu}
\address{Bard College at Simon's Rock, 84 Alford Road, Great Barrington, MA 01230}

\author{Alastair N. Fletcher}
\email{afletcher@niu.edu}
\address{Department of Mathematical Sciences, Northern Illinois University, DeKalb, IL 60115-2888, USA}

\author{Daniel A. Nicks}
\email{dan.nicks@nottingham.ac.uk}
\address{School of Mathematical Sciences, The University of Nottingham, University Park, Nottingham, NG7 2RD, UK}

\begin{abstract}
We construct a quasiregular mapping in $\R^3$ that is the first to illustrate several important dynamical properties: the quasi-Fatou set contains wandering components; these quasi-Fatou components are bounded and hollow; and  the Julia set has components that are genuine round spheres. The key tool in this construction is a new quasiregular interpolation in round rings in $\R^3$ between power mappings of differing degrees on the boundary components. We also exhibit the flexibility of constructions based on these interpolations by showing that we may obtain quasiregular mappings which grow as quickly, or as slowly, as desired.
\end{abstract}

\maketitle

\section{Introduction}

\subsection{Background}
Let $f: \mathbb{C} \rightarrow \mathbb{C}$ be a non-constant entire function, and denote by $f^n$ the function~$f$ composed with itself $n$ times. The Julia set of $f$, denoted $J(f)$, is the set of all points~$z$ where the orbit, $\{f^n(z)\}$, exhibits sensitivity to initial conditions. Its complement is called the Fatou set, denoted $F(f)$. The Fatou set is open, and its connected components are called Fatou components.

The nature of the singularity at infinity for $f$ gives a key dichotomy in the study of entire functions. If $\infty$ is a pole, then $f$ is a polynomial. Otherwise, $f$ has an essential singularity there, and $f$ is called a transcendental entire function.	
Transcendental entire functions exhibit many dynamical phenomena that do not exist in the polynomial setting. For example, it is possible for transcendental entire functions to have \textit{multiply connected wandering domains}, and entire functions with this property are one of the main motivations for this article. 

 In \cite{Baker63}, Baker exhibited the first example of a transcendental entire function with a multiply connected Fatou component, given by an infinite product of the form
\begin{equation*}
f(z) = Cz^2\prod_{n=1}^{\infty}\left(1+\frac{z}{R_n}\right).
\end{equation*}
Here, $C$ is a large positive constant, and $\{R_n\}$ is a rapidly increasing sequence. One can think of the $R_n$ as radii where the function $f$ increases its degree from $2+(n-1)$ to $2+n$.  Moreover, a key part of Baker's analysis uses the fact that on annuli of the form $A(R_n^2, R_{n+1}^{1/2})$, the function $f(z)$ is a small perturbation of a monomial of degree $2+n$. He later showed, in \cite{Baker76}, that $f$ has a \emph{bounded} multiply connected Fatou component $U$, and thus $f^n(U)$ is a sequence of wandering domains. This was the first example of wandering domains for entire functions. In fact, Baker showed \cite{Baker75} that \emph{all} multiply connected Fatou components of transcendental entire functions are bounded, and so are wandering domains on which the iterates tend to infinity.

Bergweiler, Rippon, and Stallard gave a comprehensive characterization of the internal dynamics in multiply connected wandering domains in \cite{BRS13}. Among other things, they demonstrated that Baker's example has several features that generalize. They showed that every sequence of multiply connected wandering domains contains large round annuli where the function is very close to a monomial \cite[Theorem 5.2]{BRS13}. These annuli eventually ``absorb" the orbits of all points in the multiply connected wandering domains. Therefore, the orbit of a point in a multiply connected Fatou component eventually lands in this absorbing annulus, and thereafter, the dynamics behave roughly like that of iterating monomials of increasing degree.
	
Bergweiler, Rippon, and Stallard were additionally able to modify Baker's infinite product construction to produce a variety of functions having multiply connected wandering domains with interesting topological features, see \cite[Section 10]{BRS13}. Bishop used an infinite product construction in \cite{Bis18} to construct an entire function whose Julia set has Hausdorff dimension 1, and his results were generalized by the first named author in \cite{Bur21} to construct Julia sets with fractional packing dimension. 

Infinite products are not the only tools that allow one to construct a function modeled by monomials with increasing degree on large annuli. Kisaka and Shishikura \cite{KS08} used a technique known as quasiconformal surgery to construct an entire function with a multiply connected wandering domain that is a topological annulus. All previously constructed examples where the connectivity was known had infinitely many complementary components. The first author and Lazebnik gave another approach to constructing functions with multiply connected wandering domains in \cite{BL23} by modifying a construction of Bishop known as quasiconformal folding \cite{Bis15}. The construction in \cite{BL23}  motivates many of the ideas in this article.

Inspired by the above results,  our attention here turns to analogous ideas in higher dimensions.
Quasiregular dynamics is the natural generalization of complex dynamics to higher real dimensions, and thus is the appropriate setting for our study. There has been a substantial amount of work studying the Julia set $J(f)$ and its complement the quasi-Fatou set $QF(f)$ for transcendental type quasiregular maps $f:\R^n \to \R^n$, for $n\geq 2$. 
This was initiated by Bergweiler and the third named author in \cite{BN14}. Of particular relevance here is the paper \cite{NS17} of the third named author and Sixsmith, where it was shown that in every dimension at least $2$, there exists a quasiregular map $f$ for which $QF(f)$ has a component which is hollow \cite[Theorem 1.1]{NS17}.
Here, a domain $U \subset \R^n$ is called hollow if there exists a bounded component of $\R^n \setminus U$, and is called full if there is no bounded complementary component. In dimension $2$, a domain $U$ is hollow if and only if $U$ is multiply connected, but in higher dimensions the richer topology means this is no longer true --- consider, for example, a solid torus.

A significant open question in quasiregular dynamics is whether every hollow quasi-Fatou component must be bounded. If so, this would be analogous to Baker's result in the complex setting that every multiply connected Fatou component is bounded, which he proved over ten years after giving the first example of such Fatou components. An affirmative answer to the above question would imply that, for any transcendental type quasiregular map $f$, the Julia set $J(f)$ is perfect and coincides with the boundary of the fast escaping set, and that $J(f) = J(f^p)$ for $p\in\N$, see \cite{NS17}. These properties
are familiar in complex dynamics but have yet to be established in general in the quasiregular setting. 

Any quasi-Fatou component that is hollow and bounded must be a wandering domain. Moreover, the results of \cite{NS17, NS17b} establish a strong analogy between the structure and behaviour of such
components and that of multiply connected wandering domains in complex dynamics.

\subsection{Statement of results}

It was left open in \cite{NS17} whether the hollow components of $QF(f)$ found there were bounded or unbounded. In this paper, we address this question via the following result.

\begin{theorem}
\label{thm:hollow}
There exists a quasiregular map $f:\R^3 \to \R^3$ for which $QF(f)$ contains bounded, hollow components. Moreover, there exist a sequence of quasi-Fatou components $(\Omega_k)_{k=1}^{\infty}$ and a sequence of spheres $(\Gamma_k)_{k=1}^{\infty}$ centered at the origin and with radii tending to infinity such that, for $k\in \N$:
\begin{enumerate}[(a)] 
\item the $\Gamma_k$ are components of $J(f)$,
\item $f$ maps $\Gamma_k$ onto $\Gamma_{k+1}$, 
\item $\Omega_k$ is a bounded, hollow component of $QF(f)$,
\item $\Gamma_k$ and $\Gamma_{k+1}$ are boundary components of $\Omega_k$.
\end{enumerate}
\end{theorem}

We re-iterate that this is the first construction of a quasiregular map in higher dimensions for which $QF(f)$ is known to contain bounded, hollow components, and that these are necessarily wandering domains. In fact, this is the first example of a transcendental type quasiregular map of~$\R^n$, with $n\ge3$, for which it is known that the quasi-Fatou set has more than one component.  

The key step in our construction is an interpolation for quasiregular power maps of different degrees. The technique is inspired by the construction of Lazebnik and the first named author which gives an interpolation between $z^m$ and $z^n$ in planar annuli \cite{BL23}. It is well known that interpolation in ring domains in the bi-Lipschitz and quasiconformal categories is possible by Sullivan's Annulus Theorem, see \cite{Sullivan}, or \cite{TV} for a quantitative version. However, for non-injective maps, general interpolation results are rather rare: the Berstein-Edmonds extension theorem in dimension three~\cite{BE79} has been used several times in uniformly quasiregular dynamics, see \cite{FS22,FSV22,FW15}. See also~\cite{PW24} which includes, among many other things, work along a similar theme in dimension four.

Our main interpolation result concerns quasiregular maps $p_d:\R^3\to\R^3$ analogous to complex power maps $z^d$. We mention here only that $|p_d(x)|=|x|^d$ and that $p_d$ has degree $d^2$, deferring a fuller introduction of these maps to section~\ref{prelims}.

To set some notation, for $y\in \R^n$ and $0<r<s$, we write $B(y,r)$ for the open ball $\{x\in \R^n : |x-y| <r \}$, and $A(r,s)$ for the closed ring $\{ x \in \R^n : r \leq |x| \leq s \}$.

\begin{theorem}
\label{thm:interp}
Let $d\in\N$ be odd. There exists a quasiregular map $P: A(1,e^{1/d})\to \overline{B(0,e^3) }$ that equals $p_d$ when $|x|=1$ and equals $p_{3d}$ when $|x|=e^{1/d}$. The dilatation $K(P)$ is independent of $d$.
\end{theorem}

The only similar result of which the authors are aware is a construction by Drasin and Sastry~\cite{DS03} (used in~\cite{NS17}), which gives a quasiregular interpolation in ring domains in $\R^n$ between piecewise linear power maps of different degrees. However, the construction in \cite{DS03} is somewhat technical and is achieved by increasing the degree in one coordinate direction at a time. Our construction is more amenable for the applications we have in mind. We restrict our construction to $\R^3$ and while we would expect similar results to hold in higher dimensions, we leave this open.

In section~\ref{sec5}, we illustrate the flexibility of our methods by using Theorem~\ref{thm:interp} to construct quasiregular maps that grow arbitrarily quickly or arbitrarily slowly.

The authors would like to thank Allyson Hahn for creating the figures.

\section{Preliminaries}\label{prelims}

\subsection{Quasiregular maps}\label{sect:qrmaps}

We refer to Rickman's monograph \cite{Rickman} for foundational material on quasiregular maps. The inner dilatation and outer dilatation of $f$ are denoted by $K_I(f)$ and $K_O(f)$ respectively, and the maximal dilatation is $K(f) = \max \{ K_I(f) , K_O(f) \}$.

A quasiregular map $f:\R^n \to \R^n$ is said to be of transcendental type if $f$ has an essential singularity at infinity. 
The Julia set $J(f)$ of $f$ is then defined to be the set of all $x$ such that
\begin{equation*} 
\operatorname{cap} \left(\R^n\backslash \bigcup\limits^\infty_{k=1}f^k(U)\right)=0 
\end{equation*}
for every neighbourhood $U$ of $x$. Here $\operatorname{cap} E=0$ denotes that $E$ is a set of conformal capacity zero. We refer to \cite{BN14} for a fuller discussion of this, noting here only that $J(f)$ comprises points with a certain `blowing up' property.

The quasi-Fatou set $QF(f)$ is $\R^n \setminus J(f)$. By \cite[Theorem 1.1]{BN14}, $J(f)$ is non-empty and, in fact, guaranteed to be an infinite set. Both the Julia set and quasi-Fatou set are completely invariant; that is, $x\in J(f)$ if and only if $f(x) \in J(f)$.

\subsection{Zorich maps}\label{sec:Zorich}

The  Zorich maps are quasiregular maps which generalize the complex exponential function. There is not an obvious canonical choice of Zorich map,  but rather a class of mappings. We briefly describe how to define a Zorich map $\mathcal{Z} \colon \R^3 \to \R^3 \setminus \{ 0 \}$ based on an initial choice of bi\nobreakdash-\hspace{0pt}Lipschitz map $h$ from the square $[-1,1]^2$ to the upper hemisphere $\{ y = (y_1, y_2,y_3) : |y|=1, \, y_3\ge0\}$ in $\R^3$. One example of such a map is given by (cf.~\cite[I.3.3]{Rickman})
 
\begin{equation}
\label{eq:zorich1} 
h(x_1,x_2) = \left ( \frac{ x_1 \sin (\pi M(x_1,x_2)/2 )}{\sqrt{x_1^2+x_2^2} } , \frac{ x_2 \sin (\pi M(x_1,x_2)/2) }{\sqrt{x_1^2+x_2^2} } , \cos (\pi M(x_1,x_2)/2 ) \right ),
\end{equation}
where $M(x_1,x_2) = \max \{ |x_1| , |x_2| \}$.  We do not assume that $h$ has the specific form \eqref{eq:zorich1}, but we shall impose the additional condition that
\begin{equation}\label{eq:h(0)}
    h(0,0) = (0,0,1).
\end{equation}
In the beam $[-1,1]^2 \times \R$, we define the Zorich map $\mathcal{Z}$ by
\begin{equation}
\label{eq:zorich2} 
\mathcal{Z} (x_1,x_2,x_3) = e^{x_3} h(x_1,x_2).
\end{equation}
The image of the beam under $\mathcal{Z}$ is the closed upper half-space with the origin removed. By repeatedly reflecting in sides of beams in the domain, and in the plane $\{x_3 = 0\}$ in the range, we extend the definition to obtain our Zorich map $\mathcal{Z} \colon \R^3 \to \R^3 \setminus \{ 0 \}$. 
For more details on Zorich maps, including a justification of their quasiregularity, we refer to \cite[\S6.5.4]{IM01}.

If $H\subset \R^3$ is a plane, then we denote by $R_H$ the reflection in $H$. The map $\mathcal{Z}$ is strongly automorphic with respect to the group $G$ of isometries obtained by composing an even number of reflections of the form $R_{\{x_i = n\}}$ with $i=1,2$ and $n$ odd. Here, strongly automorphic means that 
\begin{itemize}
\item for all $R\in G$, we have $Z\circ R = Z$; and
\item if $\mathcal{Z}(a) = \mathcal{Z}(b)$, then there exists $R\in G$ such that $a= R(b)$.
\end{itemize}
In particular, $\mathcal{Z}$ has periods $(4,0,0)$ and $(0,4,0)$.
Moreover, the branch set of $\mathcal{Z}$ is $\mathcal{B}(\mathcal{Z}) = \{ (m,n,x_3) : m,n \text{ odd integers, } x_3 \in \R \}$.

Throughout this paper we will denote by $\mathcal{Z}$ a Zorich map as above with associated automorphism group $G$. The condition \eqref{eq:h(0)} means that $\mathcal{Z}$ maps the $x_3$-axis into itself.

\subsection{Power maps}

If $u:\R^3 \to \R^3$ is an expanding conformal map which satisfies $u(0) = 0$ and $uGu^{-1} \subset G$, then by \cite[Theorem 21.4.1]{IM01} the Schr\"oder equation $p \circ \mathcal{Z}  = \mathcal{Z} \circ u$ has a unique uniformly quasiregular solution $p$.

For any integers $d$ and $n$ we have the identity
\begin{equation}
d \cdot R_{\{x_i=n\}}(x) = R_{\{x_i=dn\}}(dx). \label{eqn:dR}
\end{equation}   
Now let $d$ be a positive odd integer and take $u(x)=dx$. Then it follows from \eqref{eqn:dR} that $uGu^{-1} \subset G$ and so there is a unique solution to the Schr\"oder equation that we denote by $p_d$.

These maps $p_d$ are uniformly quasiregular analogues of power maps $z^d$ and were first studied by Mayer \cite{Mayer97}. See also \cite{FM20} for more on uniformly quasiregular solutions of Schr\"oder equations. We note that $|p_d(x)| = |x|^d$ and that $p_d$ has degree $d^2$. The former follows from the Schr\"oder equation and \eqref{eq:zorich2}, since if $x=\mathcal{Z}(y)$ then
\[ |p_d(x)| = |\mathcal{Z}(dy)| = e^{dy_3} = |\mathcal{Z}(y)|^d = |x|^d.\]

\section{Quasiregular Interpolation}

In this section, we prove Theorem \ref{thm:interp}.
To start, let $U= [-1,1]^2\times [0,1]$, let $V=[-3,3]^2\times [0,3]$ and denote the lower half of the vertical faces of $U$ by
\[ F_0 = \{-1,1\}\times [-1,1] \times [0,\tfrac12] \ \cup \ [-1,1]\times \{-1,1\} \times [0,\tfrac12]. \]

\begin{lemma}
\label{lem:1}
There exists a continuous, orientation-preserving, piecewise linear map  ${\alpha'\colon U\to V}$ such that 
\[ \alpha'(x) = 
\begin{cases} x, & \mbox{on the base } [-1,1]^2\times\{0\} \\
3x, & \mbox{on the top } [-1,1]^2\times\{1\}
\end{cases} \]
and $\alpha'$ maps the flaps $F_0$ into $\{x_3=0\}$ while the top half of the vertical faces of $U$ are mapped to the vertical faces of $V$. See Figure \ref{fig:1}.
\end{lemma}

\begin{figure}[ht]
\begin{center}
\includegraphics[width=6in]{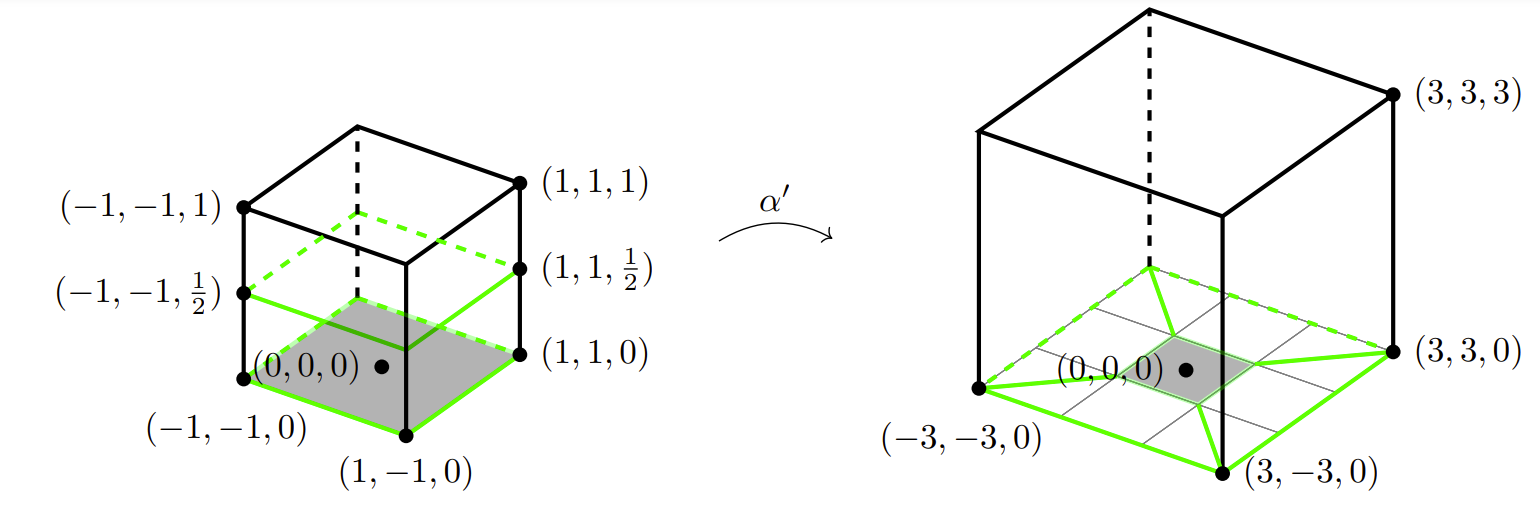}
\caption{The map $\alpha'$.}
\label{fig:1}
\end{center}
\end{figure}

\begin{proof}
Let $T_1$ be the triangle with vertices $(0,0,0)$, $(1,1,0)$ and $(0,1,0)$. We first work on the prism $\{ T_1 + (0,0,t) : t\in [0,1] \}$. We split the rectangle $[0,1]\times\{1\}\times [0,\frac12]$ into two triangles $T_2, T_3$ with vertices $(0,1,0), (1,1,0), (0,1,\tfrac12)$ and $(1,1,0),(0,1,\tfrac12),(1,1,\tfrac12)$ respectively. We also split the rectangle $[0,1]\times\{1\}\times [\frac12,1]$ into two more triangles $T_4, T_5$ with vertices $(0,1,\tfrac12),(1,1,\tfrac12),(0,1,1)$ and $(0,1,1),(1,1,\tfrac12),(1,1,1)$ respectively, see Figure \ref{fig:2}. The prism now has a simplicial decomposition into five simplices $\Delta_1, \ldots, \Delta_5$ where $\Delta_j$ is the convex hull of $T_j$ and the vertex $(0,0,1)$.

\begin{figure}[ht]
\begin{center}
\includegraphics[width=6in]{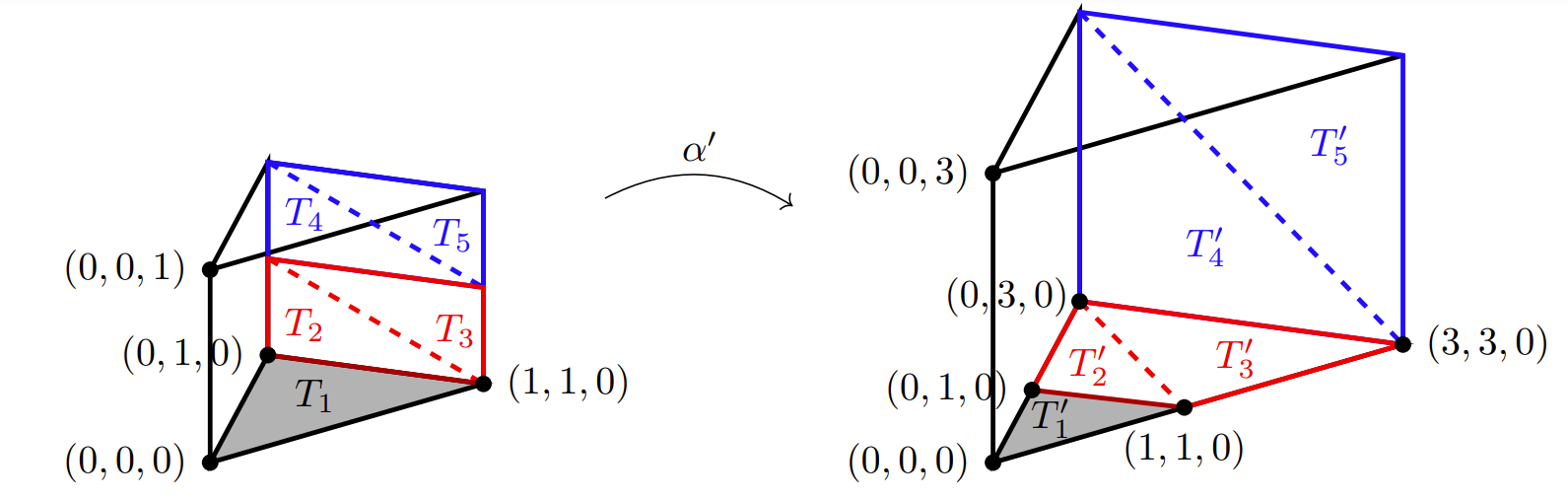}
\caption{The initial prism.}
\label{fig:2}
\end{center}
\end{figure}

In the range, we take a triangle $T_1'$ with the same vertices as $T_1$; triangle $T_2'$ with vertices $(0,1,0), (1,1,0),$ and $(0,3,0)$; triangle $T_3'$ with vertices $(1,1,0),(0,3,0),$ and $(3,3,0)$; triangle $T_4'$ with vertices $(0,3,0),(3,3,0),$ and $(0,3,3)$; and triangle $T_5'$ with vertices $(0,3,3), (3,3,0),$ and $(3,3,3)$. We form the simplices $\Delta'_j$ with face $T'_j$ and vertex $(0,0,3)$. Then there exist linear maps $\Delta_j\to \Delta'_j$ that respect vertices. Note that the map on $T_1$ is the identity, and that the map on the top triangle of the prism is $x\mapsto 3x$. We define the map $\alpha'$ on the prism by these five linear maps, noting that this is continuous, orientation-preserving and maps correctly part of one of the vertical faces of $U$.
An explicit expression for $\alpha'$ is not necessary for the argument, but is easily computed as:
\begin{itemize}
\item on $\Delta_1$, $\alpha'(x_1, x_2, x_3) = (x_1, x_2, 3x_3)$
\item on $\Delta_2$, $\alpha'(x_1, x_2, x_3) = (x_1, 5x_2+4x_3-4, 3-3x_2)$ 
\item on $\Delta_3$, $\alpha'(x_1, x_2, x_3) = (3x_1+2x_2+4x_3-4, 5x_2+4x_3-4, 3-3x_2) $ 
\item on $\Delta_4 \cup \Delta_5$, $\alpha'(x_1, x_2, x_3) = (3x_1, 3x_2, 6x_3 - 3)$.
\end{itemize}

We extend $\alpha'$ to all of $U$ by reflections in the domain and range. Specifically, if $H$ is any of the four planes $\{x_1=0\}$, $\{x_2=0\}$ or $\{x_1=\pm x_2\}$, then $\alpha'$ satisfies $\alpha'(R_H(x)) = R_H(\alpha'(x))$. Noting that such $R_H$ commute with $x\mapsto 3x$, we see that $\alpha'$ has the required properties, establishing Lemma \ref{lem:1}.
\end{proof}

Next, we shall further extend our mapping. Define a grid of flaps
\[ F = \bigcup_{i\in\Z} \R \times \{2i-1\} \times [0,\tfrac12) \ \cup \ 
\bigcup_{j\in\Z} \{2j-1\} \times \R  \times [0,\tfrac12) \]
and define an orientation-preserving piecewise linear map
\[ \alpha\colon \left(\R^2\times [0,1]\right)\setminus F \to \R^2 \times [0,3] \]
by extending $\alpha'$ by reflections in the domain and range. Specifically, we take
\begin{equation}
  \alpha(R_{\{x_i=n\}}(x)) = R_{\{x_i=3n\}}(\alpha(x)) \label{defn_alpha}
\end{equation} 
  for $i\in\{1,2\}$ and $n$ an \emph{odd} integer. Note that this map $\alpha$ is well-defined and continuous (on the stated domain) because $\alpha'$ mapped the rectangles above the flaps $F_0$ to the vertical faces of $V$.

\begin{lemma}
\label{lem:2}
The mapping $\alpha$ has the following properties, in which $G$ is the automorphism group of the Zorich map $\mathcal{Z}$.
\begin{enumerate}
\item $\alpha(u, v, 1) = (3u, 3v, 3)$.
\item If $n,m\in\Z$ and $|u-2n|<1$ and $|v-2m|<1$, 
then $\alpha(u, v, 0) = (u+4n, v+4m, 0)$.
\item If $R\in G$ then there exists $R'\in G$ such that $\alpha(R(x)) = R'(\alpha(x))$.
\item If $\zeta$ lies on exactly one flap in $F$ then there exist two points $\zeta^+$ and $\zeta^-$ in $\{x_3=0\}$, symmetric across some plane of the form $\{x_1 = \mbox{odd integer}\}$ or $\{x_2 = \mbox{odd integer}\}$ such that ${\alpha(x)\to\zeta^+}$ as $x\to\zeta$ from one side of the flap and $\alpha(x)\to\zeta^-$ as $x\to\zeta$ from the other. If $\zeta$ lies on the intersection of two flaps in $F$, then there exist four points in $\{x_3=0\}$ symmetric across a plane $\{x_1 = \mbox{odd}\}$ and a plane $\{x_2 = \mbox{odd}\}$ such that $\alpha(x)$ tends to one of these points as $x\to\zeta$ from each of the four square beams around $\zeta$.
\end{enumerate}
\end{lemma}

\begin{proof}

Recalling from section \ref{sec:Zorich} that $G$ is exactly compositions of even numbers of reflections of form $R_{\{x_i=n\}}$ with $i=1,2$ and $n$ odd, we see that (iii) follows immediately from \eqref{defn_alpha}.

To prove (iv), note that if $\zeta$ lies on a flap then $\alpha(x)$ certainly tends to some limit as $x\to\zeta$ from within one beam. The symmetry with respect to limits from different sides or beams follows from~\eqref{defn_alpha}.

We next take $x=(u, v, 1)$ and aim to show (i). There exists $y\in [-1,1]^2\times\{1\}$ such that $x=S(y)$ where $S$ is some composition of up to four reflections of the form $R_{\{x_i=n\}}$ with $i=1, 2$ and $n$ odd. Let $\tilde{S}$ be the corresponding composition in which each reflection plane is multiplied by~$3$; for example, if $S=R_{\{x_2=3\}}R_{\{x_1=-5\}}$ then $\tilde{S}=R_{\{x_2=9\}}R_{\{x_1=-15\}}$. Now, using~\eqref{defn_alpha}, Lemma~\ref{lem:1}, and~\eqref{eqn:dR} with $d=3$, we obtain
\[ \alpha(x) = \alpha(S(y)) = \tilde{S}(\alpha(y)) = \tilde{S}(3y) = 3S(y) = 3x. \]

Note that property (ii) holds when $n=m=0$ by Lemma \ref{lem:1}. Hence by induction it will suffice to show that if property (ii) holds on one square, then it holds on any neighbouring square. So we suppose that (ii) holds for some $n,m$ and without loss of generality try to prove it for $n+1, m$. Thus we take $u$, $v$ with $|u - 2(n+1)|<1$ and $|v - 2m|<1$, and note that 
\[ R_{\{x_1=2n+1\}}(u, v, 0) = (4n+2 - u, v, 0) \]
lies in the ``$n, m$'' square $(2n-1, 2n+1) \times (2m-1, 2m+1)\times\{0\}$. So, using \eqref{defn_alpha} and the inductive hypothesis, we get
\begin{align*}
\alpha(u, v, 0) & = \alpha(R_{\{x_1=2n+1\}}(4n+2 - u, v, 0))\\
&= R_{\{x_1=6n+3\}}(\alpha(4n+2 - u, v, 0))\\
&= R_{\{x_1=6n+3\}}(8n+2-u, v+4m, 0)) \\
&= (u + 4(n+1), v+4m, 0)
\end{align*}
as required. This completes the proof of Lemma \ref{lem:2}.
\end{proof}

We are now ready to define a function which is almost the interpolating map $P$ that we seek, except that it has discontinuities across flaps $\mathcal{Z}(\frac{1}{d} F)$ attached to the unit sphere, see Figure \ref{fig:3}. We will later use the fact that Lemma \ref{lem:2}(iv) tells us that these jumps happen in a simple way.

\begin{figure}[ht]
\begin{center}
\includegraphics[width=3in]{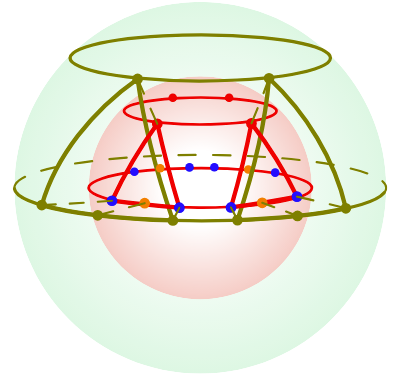}
\caption{The image of some of the flaps for the Zorich map $\mathcal{Z}$ based on \eqref{eq:zorich1}. The red, blue and orange dots indicate the images of the corners of the grid and the green dots indicate the corners of the corresponding flaps.}
\label{fig:3}
\end{center}
\end{figure}

Define $g\colon A(1,e^{1/d})\setminus \mathcal{Z}(\tfrac{1}{d} F)\to A(1,e^3)$ by
\begin{equation} 
\label{eq:funcg}
g(x) = (\mathcal{Z} \circ \alpha)(d \mathcal{Z}^{-1}(x)). 
\end{equation}
This is well-defined because if $\mathcal{Z}(y_1)=\mathcal{Z}(y_2)$, then the strong automorphy of $\mathcal{Z}$ with respect to $G$, together with \eqref{eqn:dR} and Lemma~\ref{lem:2}(iii), yield $R'\in G$ such that $\alpha(dy_1)=R'(\alpha(dy_2))$. It follows that $g$ is quasiregular (on the interior of the stated domain), with dilatation independent of $d$. We observe that:
\begin{itemize}
\item if $|x|=1$ then the third component of $d\mathcal{Z}^{-1}(x)$ is $0$, so by Lemma \ref{lem:2}(ii) we get that 
\[g(x) = \mathcal{Z}(d\mathcal{Z}^{-1}(x)) = p_d(x);\]
\item if $|x|=e^{1/d}$ then the third component of $d\mathcal{Z}^{-1}(x)$ is $1$, so by Lemma \ref{lem:2}(i) we get that
\[ g(x) = \mathcal{Z}(3d\mathcal{Z}^{-1}(x)) = p_{3d}(x). \]
\end{itemize}

In order to prove Theorem~\ref{thm:interp}, it remains to deal with the discontinuities of $g$ across the flaps $\mathcal{Z}(\frac{1}{d} F)$. To do this, we follow \cite{BL19} by introducing a map $\beta$ to ``close the gap'' between the images of points on opposite sides of these flaps. The image points that need to be stuck back together are symmetric across the $\{x_3=0\}$ plane. 

\begin{lemma}
\label{lem:3}
There exists a $9$-quasiconformal map 
\[ \beta \colon \{|x|>1\} \to \R^3 \setminus \{x_1^2+x_2^2\le 1, \, x_3=0\} \]
from the complement of the unit ball $\overline{B(0,1)}$ to the complement of a disc, such that $\beta$ extends continuously to $\partial B(0,1)$ with
\begin{equation}
\beta(x_1, x_2, x_3) = \beta(x_1, x_2, -x_3) \quad \mbox{when } x_1^2+x_2^2+x_3^2=1, \label{beta(-x_3)}
\end{equation} 
and such that $\beta$ is the identity outside the set
\begin{equation}
\label{eq:X} 
X:= \left( B((0,0,1), \sqrt{2}) \cup B((0,0,-1), \sqrt{2})\right) \setminus B(0,1). \
\end{equation}
\end{lemma}

\begin{proof}

Following the planar case \cite{BL19}, we define $\beta:= \mu \circ \nu \circ \mu$ where $\mu = \mu^{-1}$ is the self-inverse orientation-reversing M\"{o}bius map of $\R^3\cup\{\infty\}$ (sending $\{|x|>1\}$ to $\{x_1>0\}$) given by
\[ \mu(x_1, x_2, x_3) = \frac{2(x_1-1, x_2, x_3)}{|(x_1-1, x_2, x_3)|^2} + (1,0,0) = (1+2r(x_1-1), 2rx_2, 2rx_3), \]
where $r=r(x_1, x_2, x_3) = ((x_1-1)^2 + x_2^2 + x_3^2)^{-1}$, and 
\[ \nu\colon \{x_1>0\} \to \R^3 \setminus \{x_1\le 0, \, x_3=0\} \]
is the identity in the wedge $\{x_1\ge |x_3| \}$ and triples the angle of the remaining wedges ${\{0< x_1<|x_3|\}}$. (More precisely, using cylindrical co-ordinates, $\nu(\rho\cos\phi, x_2, \rho\sin\phi)=(\rho\cos\phi', x_2, \rho\sin\phi')$ where $\rho^2=x_1^2+x_3^2$ and for $\frac\pi4 < \pm\phi<\frac\pi2$ we take $\phi'=3(\phi\mp\frac\pi4)\pm\frac\pi4$.)  The dilatation of $\nu$ works out as for the winding map (e.g.~\cite[I.3.1]{Rickman}) with $K_O(\nu)=3^2=9$ and $K_I(\nu)=3$. The composition  $\beta = \mu \circ \nu \circ \mu$ is thus orientation-preserving and quasiconformal with $K(\beta)=9$.

A calculation shows that 
\[ \mu(\{x_1 = \pm x_3\}) = \partial B((0,0,\pm 1), \sqrt{2}) \]
from which it can be seen that $\beta$ is the identity outside $X$. When $x_1^2 + x_2^2 + x_3^2 = 1$ we get that $r = \frac12(1-x_1)^{-1}$ and so 
\begin{equation}
\mu(x_1, x_2, x_3) = \left(0, \frac{x_2}{1-x_1}, \frac{x_3}{1-x_1}\right); \label{mu_when_|x|=1}
\end{equation} 
thus for such points $(\nu\circ\mu)(x_1, x_2, x_3) = (\nu\circ\mu)(x_1, x_2, -x_3) = \left(-\frac{|x_3|}{1-x_1}, \frac{x_2}{1-x_1}, 0\right)$, which yields \eqref{beta(-x_3)}. The formula \eqref{mu_when_|x|=1} also serves to help justify the claim that $\mu$ maps $\{|x|>1\}$ to $\{x_1>0\}$, as well as showing that $\mu$ fixes pointwise the circle $\{(x_1,x_2,x_3): x_1=0, \, x_2^2+x_3^2=1\}$.
\end{proof}

We are now in a position to complete the proof of Theorem \ref{thm:interp}.

\begin{proof}[Proof of Theorem~\ref{thm:interp}]
First, note that by \eqref{eq:X}, we have $X \subset A(1,e^3)$. Let $E = A(1,e^{1/d})\setminus \mathcal{Z}(\tfrac{1}{d} F)$ be the domain of the function $g$ defined in \eqref{eq:funcg}. The pre-image of $X$ under $g$ consists of $2(3d)^2$ components, $2d^2$ of which neighbour the sphere $\{|x|=1\}$ with the remaining $16d^2$ neighbouring the flaps $\mathcal{Z}(\frac{1}{d}F)$. Let $W\subset E$ denote the union of those components neighbouring the flaps. For $x\in E$, we define 
\[ P(x) = \begin{cases} \beta \circ g(x), & x \in W, \\
g(x), & x \in  E\setminus W.
\end{cases} \]
This is continuous on $\partial W \cap E$ since $\beta(x)=x$ on $\partial X$. The map $P$ is quasiregular on $E$ with dilatation independent of $d$ as it is a composition of such maps. Moreover, $P$ extends continuously across the flaps $\mathcal{Z}(\tfrac{1}{d} F)$.
 This follows from the definition of $g$, Lemma \ref{lem:2}(iv) and \eqref{beta(-x_3)}: as $x$ approaches a flap from different sides, $g(x)$ approaches two values symmetric across the $\{x_3=0\}$ plane, and these two values are mapped to a single value by $\beta$. The extended map $P$ is quasiregular on $A(1,e^{1/d})$ (with the same dilatation) since it is locally Lipschitz by construction and the flaps have measure zero.

When $|x|=1$ or $|x|=e^{1/d}$ we have that $x\notin W$, so we obtain that $P$ interpolates between $p_d$ and $p_{3d}$ as we have already shown this for $g$. 
\end{proof}

Denote by $L \subset \R^3$ the half-line $\{(0,0,t): t\geq 0\}$.

\begin{lemma}
\label{lem:L1}
The function $P$ in the proof of Theorem~\ref{thm:interp} maps $L\cap A(1,e^{1/d})$ homeomorphically into~$L$.
\end{lemma}

\begin{proof}
The map $\alpha'$ from Lemma \ref{lem:1} maps the line segment from $(0,0,0)$ to $(0,0,1)$ homeomorphically onto the line segment from $(0,0,0)$ to $(0,0,3)$, see Figure \ref{fig:2}. Also, the Zorich map  $\mathcal{Z}$ maps the $x_3$-axis homeomorphically into $L$ by \eqref{eq:h(0)} and \eqref{eq:zorich2}. Then it follows by \eqref{eq:funcg} that $g$ maps $L\cap A(1,e^{1/d})$ homeomorphically into $L$. (In fact, $g(0,0,t)=(0,0,t^{3d})$ for $1\le t \le e^{1/d}$.) As $L$ avoids the set $W$ consisting of neighbourhoods of the flaps, it follows that $P$ has the same property.
\end{proof}

\section{Spherical Julia components and hollow quasi-Fatou components}

In this section we will construct the quasiregular map which has the dynamical properties claimed in Theorem \ref{thm:hollow}. 

\subsection{Constructing the map}

The first step is to generalize the interpolation map from Theorem~\ref{thm:interp} to other rings.

\begin{theorem}
\label{thm:interp_gen}
Let $d\in\N$ be odd, and let $R>0$ and $c >0$. There exists a quasiregular map 
\[ F_{R,d,c}: A(R,Re^{1/d})\to \{|x|\le  c'R^{3d}e^3\} \] that equals $c\cdot p_d$ when $|x|=R$ and equals $c' \cdot p_{3d}$ when $|x|=Re^{1/d}$. The dilatation is independent of $d, c$, and $R$. We have
\begin{equation}
\label{eq:c'}
 c' = c \cdot R^{-2d}. 
\end{equation}
\end{theorem}

\begin{proof}
Define the dilations $\tau(x) = R^{-1}x$ and $\sigma(x) = c' R^{3d} x$. Let $P: A(1,e^{1/d}) \rightarrow \{|x| \le e^3\}$ be the quasiregular mapping from Theorem \ref{thm:interp}. We claim that 
\begin{equation}
F_{R,d,c}(x) := \sigma \circ P \circ \tau(x)
\end{equation}
is the desired mapping. This follows from the property that for all $R > 0$, we have $p_d(Rx) = R^dp_d(x)$. Indeed, if $|x| = R$, then by \eqref{eq:c'}
\begin{align*}
F_{R,d,c}(x) &= \sigma \circ P\left(\frac{1}{R}x\right) \
	= \sigma \circ p_d\left(\frac{1}{R}x\right) 
	= c'R^{3d} \cdot p_d\left(\frac{1}{R}x\right) 
	= c'R^{2d} \cdot p_d(x) 
	= c \cdot p_d(x).
\end{align*}
On the other hand, if $|x| = Re^{1/d}$, then 
\begin{align*}
F_{R,d,c}(x) &= \sigma \circ P\left(\frac{1}{R}x\right) 
= \sigma \circ p_{3d}\left(\frac{1}{R}x\right) 
= c'R^{3d}\cdot p_{3d}\left(\frac{1}{R}x\right) 
= c' \cdot p_{3d}(x), 
\end{align*}
as desired. 
\end{proof}

Let $R_1>2$, $c_1 = 1$, and define
\begin{equation}
\label{eq:d_n}
d_n = 3^n.
\end{equation}
Next, proceeding inductively, suppose that we have constructed $c_1,\dots, c_n$ and $R_1,\dots,R_n$. Define
\begin{equation}
\label{eq:parameters}
c_{n+1} = c_n R_n^{-2d_{n}} \ \ \text{   and   } \ \ R_{n+1} = c_{n+1}(2R_n)^{d_{n+1}}.
\end{equation}
Finally, define $S_n = e^{1/d_n}R_n$. The function we consider is, for $n\geq 1$,
\begin{equation}
\label{eq:QRDef}
f(x) = 
    \begin{cases}
    p_{d_1}(x) & \text{if } |x| \leq R_1 \\
    F_{R_n,d_n,c_n}(x) & \text{if } R_n \leq |x| \leq S_n \\
    c_{n+1} \cdot p_{d_{n+1}}(x) & \text{if } S_n \leq |x| \leq R_{n+1}
    \end{cases}
\end{equation}
Then $f$ is a quasiregular mapping by Theorem \ref{thm:interp_gen}. Note that we have
\begin{equation}
R_{n+1} =  c_{n+1} |p_{d_{n+1}}(x)| = |f(x)| \quad \mbox{when } |x|=2R_n.
\end{equation}

Recalling that $L$ is the non-negative $x_3$-axis, we make the following observation for later purposes.
\begin{lemma}
\label{lem:L2}
The map $f$ sends $L$ homeomorphically onto $L$.
\end{lemma}

\begin{proof}
For any odd integer $d$, the definition of the power map $p_d$ via the solution of the Schr\"oder equation $p_d\circ \mathcal{Z} = \mathcal{Z} \circ d$ implies that $p_d$ maps $L$ homeomorphically into $L$. Thus the same is true for any of the scaled power maps in the definition of $f$ from \eqref{eq:QRDef}. Moreover, as $F_{R_n,d_n,c_n}$ is obtained from $P$ by pre- and post-composing by dilations, it follows from Lemma \ref{lem:L1} that $F_{R_n,d_n,c_n}$ maps $L\cap A(R_n, S_n)$  homeomorphically into $L$. As $f$ is continuous, the lemma follows.
\end{proof}

\subsection{Covering properties}

It will be useful to know that $R_{n+1}$ is much larger than the previous term in the sequence.

\begin{lemma}
\label{lem:formula2}
For all $n \geq 1$ we have
\begin{equation*}
R_{n+1} \geq 2^{d_{n+1}} R_n^{d_{n-1}}.
\end{equation*}
\end{lemma}

\begin{proof}
    It is clear that 
    \[ R_2 = c_1R_1^{-2d_1}(2R_1)^{d_2} = 2^{d_2}R_1^{d_1} \ge R_1^3.\]
    Let $n\ge 2$ and, for the sake of induction, suppose that $R_{j+1} \geq 2^{d_{j+1}} R_j^{d_{j-1}}$ for $j = 1,\ldots,n-1$. Then in particular $R_{j+1}\ge R_j^3$ for $j=1,\ldots, n-1$, and hence $R_n\ge R_j^3$. Now, using this and \eqref{eq:parameters}, 
    \begin{align*}
R_{n+1} &= c_{n+1} (2R_n)^{d_{n+1}} \\
	&= 2^{d_{n+1}} R_n^{d_{n+1}} \prod_{j=1}^n R_j^{-2d_j} \\
	&= 2^{d_{n+1}} R_n^{d_n} \prod_{j=1}^{n-1} R_j^{-2d_j}\\
    &\ge 2^{d_{n+1}} R_n^{d_n} \prod_{j=1}^{n-1} R_n^{-2d_{j-1}}\\
    &= 2^{d_{n+1}} R_{n}^{d_n - \sum_{j=1}^{n-1}2d_{j-1}} \\
    &\ge 2^{d_{n+1}} R_{n}^{d_{n-1}}.
\end{align*}
The lemma now follows by induction.
\end{proof}

For $k\geq 1$, define the round rings
\begin{align*}
	A_k &= A\left (\frac{1}{4}R_k,4R_k \right ) \\
	B_k &= A\left (4R_k, \frac{1}{4}R_{k+1} \right ) \\
	V_k &= A\left (\frac{3}{2}R_{k}, \frac{5}{2}R_{k}\right ).
\end{align*}

Clearly $V_k\subset A_k$.

\begin{lemma}
\label{lem:akcovering}
For $k\geq 1$, we have $B(0,4R_k) \subset f(A_k)$.
\end{lemma}

\begin{proof}
As the ring $\{ x : R_k \leq |x| \leq S_k \}$ is contained in $A_k$, and as $f$ agrees with the map $F_{R_k,d_k,c_k}$ on this ring, it follows by Theorem \ref{thm:interp_gen} that $f(A_k)$ contains the ball $B(0,c_kR_k^{d_k}e^3)$. By \eqref{eq:parameters} and Lemma \ref{lem:formula2}, we have
\begin{align*}
c_kR_k^{d_k}e^3 &= c_{k+1} R_k^{2d_k} R_k^{d_k} e^3 \\
&= c_{k+1} R_k^{d_{k+1}} e^3 \\
&= R_{k+1}2^{-d_{k+1}} e^3 \\
&\geq R_k^{d_{k-1}} e^3 > 4R_k,
\end{align*}
from which the lemma follows.
\end{proof}

\begin{lemma}
\label{lem:bk}
For $k\geq 1$, we have $f(B_k) \subset \interior B_{k+1}$.
\end{lemma}

\begin{proof}
As $S_k = e^{1/d_k}R_k < \tfrac32 R_k$, we have $f|_{B_k} = c_{k+1} \cdot p_{d_{k+1}}$. Therefore, $f$ maps the round ring $B_k$ to another round ring. When $|x| = 4R_k$, we have by \eqref{eq:parameters},
\begin{align*}
|f(x)| &= c_{k+1} |p_{d_{k+1}}(x)| \\
	&= c_{k+1} 4^{d_{k+1}} R^{d_{k+1}}_{k} \\
	&= c_{k+1} 2^{d_{k+1}} 2^{d_{k+1}} R^{d_{k+1}}_{k} \\
	&= 2^{d_{k+1}} R_{k+1} > 4 R_{k+1}.
\end{align*}
When $|x| = \frac{1}{4}R_{k+1}$, we have by \eqref{eq:parameters},
\begin{align*}
|f(x)|  &= c_{k+1} |p_{d_{k+1}}(x)| \\
	&= c_{k+1} 4^{-d_{k+1}} R^{d_{k+1}}_{k+1} \\
	&= c_{k+2} R_{k+1}^{2d_{k+1}} 4^{-d_{k+1}}R^{d_{k+1}}_{k+1} \\
        &= (c_{k+2} 2^{d_{k+2}} R_{k+1}^{d_{k+2}})\cdot (2^{-d_{k+2}}\cdot 4^{-d_{k+1}}) \\
	&= 2^{-5d_{k+1}} R_{k+2} < \frac{1}{4}R_{k+2}.
\end{align*}
It follows from these estimates that $f(B_k)$ is a round ring contained inside  $\interior B_{k+1}$.
\end{proof}

\begin{lemma}
\label{lem:vk}
For $k\geq1$, the inner boundary component of $V_k$ is mapped by $f$ into $B_k$ and the outer boundary component of $V_k$ is mapped by $f$ into $B_{k+1}$.
In particular, $A_{k+1} \subset f(V_k)$. 
\end{lemma}

\begin{proof}
Again as $S_k < \tfrac32 R_k$, we have that $f|_{V_k} = c_{k+1} \cdot p_{d_{k+1}}$. So the round ring $V_k$ is mapped to another round ring. Starting with the inner boundary component of $V_k$, when $|x| = \frac{3}{2}R_k$, we have by \eqref{eq:parameters},
\begin{align*}
|f(x)| &= c_{k+1} |p_{d_{k+1}}(x)| \\
&= c_{k+1} \left(\frac{3}{4}\right)^{d_{k+1}} (2R_{k})^{d_{k+1}} \\
&= \left(\frac{3}{4}\right)^{d_{k+1}} R_{k+1} .
\end{align*}
It follows that $|f(x)| < \tfrac14R_{k+1}$. Moreover, by Lemma \ref{lem:formula2}, we have
\begin{align*}
\left(\frac{3}{4}\right)^{d_{k+1}} R_{k+1} & \geq \left(\frac{3}{4}\right)^{d_{k+1}} 2^{d_{k+1}} R_k^{d_{k-1}} \\
&= \left(\frac{3}{2}\right)^{d_{k+1}} R_k^{d_{k-1}} \\
&>4R_k.
\end{align*}
We conclude that the inner boundary component of $V_k$ is mapped into $B_k$.

Turning now to the outer boundary component, when $|x| = \frac{5}{2}R_k$, we have by \eqref{eq:parameters},
\begin{align*}
|f(x)| &= c_{k+1} |p_{d_{k+1}}(x)| \\
&= c_{k+1} \left(\frac{5}{4}\right)^{d_{k+1}} (2R_{k})^{d_{k+1}} \\
&= \left(\frac{5}{4}\right)^{d_{k+1}} R_{k+1} \\
&> 4 R_{k+1}.
\end{align*}
By the Maximum Principle, and as $f(B_k) \subset B_{k+1}$ by Lemma \ref{lem:bk}, it follows that $f(V_k)$ cannot intersect the unbounded component of the complement of $B_{k+1}$. We conclude that the outer boundary component of $V_k$ is mapped into $B_{k+1}$.

As $A_{k+1}$ is sandwiched between $B_k$ and $B_{k+1}$, it follows that $A_{k+1} \subset f(V_k)$.
\end{proof}

For $k\geq 1$ and $n\geq 0$, define 
\begin{equation}
\Gamma_{k,n} = \{x: f^j(x) \in V_{k+j} \text{ for all } j=0,\dots,n\}.
\end{equation}
Then $\Gamma_{k,n}$ is precisely the set of all points in $V_k$ whose forward orbits remain in $V_{k+j}$ up to the $n$'th iterate. For $k\geq 1$, define
\begin{equation}\label{eq:Gk}
\Gamma_k = \bigcap_{n \geq 1} \Gamma_{k, n}.
\end{equation}
As $f|_{V_k}$ is a power map, we see that $\Gamma_{k, n}$ is also a round ring and, moreover, that for a fixed $k$
\begin{equation}
\label{eq:fngkn}
f^n(\Gamma_{k,n}) = V_{k+n}
\end{equation}
for all $n\geq 0$.
We show next that the nested intersection of these round rings is indeed a  sphere.

\begin{lemma}
\label{lem:gammak}
The set $\Gamma_k$ is a round sphere $\{x : |x| = t_k \}$.
\end{lemma}

We note that since $\Gamma_k\subset V_k$ it follows that $t_k\to\infty$ as $k\to\infty$.

\begin{proof}[Proof of Lemma~\ref{lem:gammak}]
We denote the modulus of a round ring $R=A(r,s)$ by $\operatorname{mod} R=\log(s/r)$. If the round ring $R\subset V_j$ then, since $f|_{V_j}$ is the power map $c_{j+1} \cdot p_{d_{j+1}}$, we have that $f(R) = A(c_{j+1}r^{d_{j+1}}, c_{j+1}s^{d_{j+1}})$ and hence
\[ \operatorname{mod} f(R) = \log\frac{s^{d_{j+1}}}{r^{d_{j+1}}}= d_{j+1} \operatorname{mod} R.\]

For $0\le j\le n$, the set $f^j(\Gamma_{k,n})$ is a round ring in $V_{k+j}$, so by repeatedly applying the above we see that 
\[ \operatorname{mod} f^n(\Gamma_{k,n}) = \left(\prod_{j=1}^n d_{k+j}\right)\operatorname{mod} \Gamma_{k,n}.\]
By \eqref{eq:fngkn}, $\operatorname{mod} f^n(\Gamma_{k,n}) = \operatorname{mod} V_{k+n} = \log \frac53$. It follows that
 $\operatorname{mod} \Gamma_{k,n} \to 0$ as $n\to \infty$. Since $\Gamma_{k,n}$ is a nested sequence of closed round rings, we conclude that $\Gamma_k$ is a sphere.
%
\end{proof}

\subsection{Spherical Julia components}

Next we show that $\Gamma_k$ is contained in the Julia set $J(f)$. Recalling the definition from section~\ref{sect:qrmaps}, we do this by showing directly that any neighbourhood of a point in $\Gamma_k$ has a forward orbit that ``blows up'' to cover all of $\R^3$.

\begin{lemma}
\label{lem:blowup}
Let $k\geq 1$, let $x_0 \in \Gamma_k$ and let $U$ be a neighbourhood of $x_0$. Then there exists $N\in \N$ such that for all $n\geq N$ we have $f^n(U) \supset V_{k+n}$.
\end{lemma}

\begin{proof}
By \eqref{eq:Gk} and Lemma \ref{lem:gammak}, the $(\Gamma_{k,n})_{n=1}^{\infty}$ are nested ring domains converging to ${\Gamma_k = \{x : |x| = t_k\}}$. Therefore 
\begin{equation}
\label{eq:blowup0}
\mathcal{Z}^{-1}(\Gamma_{k,n}) = \R^2 \times [ a_n,b_n]
\end{equation}
where $(a_n)_{n=1}^{\infty}$ is increasing, $(b_n)_{n=1}^{\infty}$ is decreasing and 
\begin{equation}
\label{eq:blowup1} 
\lim_{n\to \infty}a_n = \lim_{n\to \infty} b_n = \log t_k.
\end{equation}
Now, pick $y\in \mathcal{Z}^{-1}(x_0)$ and $\epsilon >0$ such that $\mathcal{Z}(B(y,2\epsilon)) \subset U$. Write $y=(y_1,y_2,\log t_k)$ and, for $n\in \N$, let
\[ D_n = [y_1-\epsilon ,y_1+\epsilon] \times [y_2-\epsilon, y_2+\epsilon] \times [a_n,b_n].\]
By \eqref{eq:blowup1} there exists $N\in \N$ such that for all $n\geq N$, we have $D_n \subset B(y,2\epsilon)$. Observe that
\begin{equation}
\label{eq:blowup2}
\mathcal{Z}(D_n) \subset U \cap \Gamma_{k,n}.
\end{equation}
By increasing $N$ if necessary, we can also assume that for all $n\geq N$, we have
\[ \prod_{i=1}^nd_{k+i} > \frac{2}{\epsilon}.\]
On $\Gamma_{k,n}$, the iterates $f^j$ are just compositions of scaled versions of power maps for $j=1,\ldots, n$. More precisely, we have
\[ f^j = (c_{k+j} \cdot p_{d_{k+j}} ) \circ \ldots \circ ( c_{k+1} \cdot p_{d_{k+1}} ).\]
For $c>0$, as $p_d$ is a solution of the Schr\"oder equation $p_d\circ \mathcal{Z} = \mathcal{Z} \circ d$, we may write
\[ c\cdot p_d(x) = \mathcal{Z} ( d\mathcal{Z}^{-1}(x) + (0,0,\log c) ),\]
so the Zorich transform of $f^n$ is an affine map. That is, for $\mathcal{Z}(x) \in \Gamma_{k,n}$, we have $f^n(\mathcal{Z}(x)) = \mathcal{Z}(L_n(x))$, where
\begin{equation}
\label{eq:blowup3}
L_n(x) = d_{k+1}d_{k+2} \ldots d_{k+n}x + (0,0,h_n),
\end{equation}
for some $h_n\in \R$. Then, for $n\geq N$, by \eqref{eq:blowup2} and \eqref{eq:blowup3}, we have
\begin{align*}
f^n(U) &\supset f^n(\mathcal{Z}(D_n)) \\
&= \mathcal{Z}(L_n(D_n))\\
&=\mathcal{Z}( d_{k+1}\ldots d_{k+n}D_n + (0,0,h_n))
\end{align*}
By construction, $d_{k+1}\ldots d_{k+n}D_n$ is a cuboid with a square base of side length at least $4$. As $\mathcal{Z}$ has periods $(4,0,0)$ and $(0,4,0)$, it follows by this, \eqref{eq:blowup0}, \eqref{eq:blowup3} again and \eqref{eq:fngkn} that
\begin{align*}
f^n(U)&\supset \mathcal{Z} ( \R^2 \times [ d_{k+1}\ldots d_{k+n}a_n + h_n ,d_{k+1}\ldots d_{k+n}b_n + h_n] )\\
&= \mathcal{Z}(L_n(\R^2 \times [a_n,b_n]))\\
&=f^n ( \mathcal{Z} (\R^2 \times [a_n,b_n]))\\
&=f^n(\Gamma_{k,n})\\
&=V_{k+n}.
\end{align*}
This completes the proof.
\end{proof}

\begin{proposition}
\label{prop:gammak}
The round sphere $\Gamma_k$ belongs to $J(f)$.
\end{proposition}

\begin{proof}
Let $x_0 \in \Gamma_k$ and let $U$ be a neighbourhood of $x_0$. Then by Lemma \ref{lem:blowup}, there exists $N\in \N$ such that for all $n\geq N$ we have by Lemma \ref{lem:vk} and Lemma \ref{lem:akcovering}
\begin{align*}
f^{n+2}(U) &\supset f^2 (V_{k+n})\\
&\supset f(A_{k+n+1}) \\
&\supset B(0,4R_{k+n+1}).
\end{align*}
Thus 
\[ \bigcup_{n=1}^{\infty} f^n(U) = \R^3\]
and we conclude that $x_0 \in J(f)$.
\end{proof}

To show that $\Gamma_k$ is in fact a component of $J(f)$, we need to turn to the quasi-Fatou set.

\begin{lemma}
\label{lem:qf1}
For $k\geq 1$, we have $B_k \subset QF(f)$.
\end{lemma}

\begin{proof}
 By Lemma \ref{lem:bk}, $f$ maps each ring $B_k$ into $\interior B_{k+1}$ and it immediately follows that $\interior B_{k+1}$ is contained in $QF(f)$ as the forward orbit of $\interior B_{k+1}$ omits a neighbourhood of the origin. By Lemma~\ref{lem:bk} again, we obtain that $B_k \subset QF(f)$.
\end{proof}

\begin{corollary}
\label{cor:gammak}
Let $k\geq 2$. Then $\Gamma_k$ is a component of $J(f)$.
\end{corollary}

\begin{proof}
For $n\in \N$, we have $f^n(\Gamma_{k,n}) = V_{k+n}$. By Lemma \ref{lem:vk}, $f$ maps the inner boundary component of $V_{k+n}$ into $B_{k+n}$ and the outer boundary component of $V_{k+n}$ into $B_{k+n+1}$. By Lemma \ref{lem:qf1} and the complete invariance of $QF(f)$, it follows that $\partial \Gamma_{k,n} \subset QF(f)$. As $\bigcap_{n=1}^{\infty} \Gamma_{k,n} = \Gamma_k$, it follows that $\Gamma_k$ is a component of $J(f)$.
\end{proof}

We next observe that the $B_k$ are contained in distinct quasi-Fatou components.

\begin{lemma}
\label{lem:qf2}
Let $k >j \geq 1$. Then $B_k$ and $B_j$ are contained in different components of $QF(f)$.
\end{lemma}

\begin{proof}
By Proposition \ref{prop:gammak}, the sphere $\Gamma_k=\{x : |x| = t_k \}$ is contained in $V_k \cap J(f)$. 

Recalling that $V_k = A(\tfrac32 R_k , \tfrac52 R_k)$ and $B_k = A( 4R_k , \tfrac14 R_{k+1})$, we see that $B_k$ is contained in the unbounded component of the complement of $\Gamma_k$ and that $B_{k-1}$ is contained in the bounded component of the complement of $\Gamma_k$.
We conclude that $B_j$ is also contained in the bounded component of the complement of $\Gamma_k$, and thus $\Gamma_k$ separates $B_j$ and $B_k$. This proves the lemma.
\end{proof}

As $B_k \subset QF(f)$, denote by $\Omega_k$ the component of $QF(f)$ which contains $B_k$. We may now give the proof of Theorem \ref{thm:hollow}. 

\begin{proof}[Proof of Theorem \ref{thm:hollow}]
Part (a) follows immediately from Corollary \ref{cor:gammak}. 

For part (b), by construction, it follows that $f(\Gamma_{k,n}) = \Gamma_{k+1,n-1}$ for $k\in \N$ and $n\geq 1$. By taking the intersection over $n$, we conclude that $f$ maps $\Gamma_k$ onto $\Gamma_{k+1}$. 

For part (c), as the quasi-Fatou component $\Omega_{k}$ contains the ring $B_k$, but is disjoint from the spheres $\Gamma_k$ and $\Gamma_{k+1}$, it must be bounded and hollow.

For part (d), we know from the proof of Corollary \ref{cor:gammak} that $\partial \Gamma_{k,n}$ is contained in $QF(f)$. If we can show that for all $k$ and $n$ the outer boundary component of $\Gamma_{k,n}$ is contained in $\Omega_k$ and the inner boundary component of $\Gamma_{k,n}$ is contained in $\Omega_{k-1}$, then we are done. To do this, we will analyze what happens on the non-negative $x_3$-axis $L$.

Note that, by \eqref{eq:fngkn}, the inner and outer boundary components of $\Gamma_{k,n}$ are mapped by $f^n$ to the inner and outer boundary components of $V_{k+n}$. Thus, by Lemma~\ref{lem:vk}, $f^{n+1}(\partial_{inner}\Gamma_{k,n})\subset B_{k+n}$ and $f^{n+1}(\partial_{outer}\Gamma_{k,n})\subset B_{k+n+1}$. Now consider the line segment $L_{k,n}\subset L$ that joins the outer boundary component of $\Gamma_{k,n}$ to the inner boundary component of $\Gamma_{k+1,n}$. We see that $f^{n+1}(L_{k,n})\subset B_{k+n+1}$ because this holds for the images of the endpoints of $L_{k,n}$ and, by Lemma~\ref{lem:L2}, the iterates of $f$ act homeomorphically on $L$. Thus $L_{k,n}\subset QF(f)$ by Lemma~\ref{lem:qf1} and the complete invariance of $QF(f)$.

So we have that $L_{k,n}\cup \partial_{outer}\Gamma_{k,n} \cup \partial_{inner}\Gamma_{k+1,n}$ is a connected subset of $QF(f)$  that intersects $B_k$, and is therefore contained in $\Omega_k$. In particular, $\partial_{outer}\Gamma_{k,n} \subset \Omega_k$ and $\partial_{inner}\Gamma_{k,n} \subset \Omega_{k-1}$.
 This completes the proof.
\end{proof}

\begin{rem}
    The proof above shows that for all $t\ge t_1$ the point $(0,0,t)$ lies either in one of the spherical Julia components $\Gamma_k=\{x:|x|=t_k\}$ or in the quasi-Fatou set $QF(f)$. On the other hand, it can be shown that each of the quasi-Fatou components $\Omega_k$ has infinitely many bounded complementary components, cf.~\cite{NS17b}.
\end{rem}

\section{Constructing quasiregular maps meeting growth conditions}\label{sec5}

Given $R>0$, $d\in \N$ and $c>0$, denote by $F_{R,d,c}$ the quasiregular map from Theorem~\ref{thm:interp_gen}.

Suppose that $\lambda_n \geq 1$ for all $n$. Construct sequences $(R_n)$ and $(S_n)$ for $n\geq 1$ as follows. Let $R_1 =1$, $S_1 = e^{1/3}$ and for $n\geq 1$, set $R_{n+1} = \lambda_n S_n$ and $S_{n+1} = e^{1/d_{n+1}}R_{n+1}$, recalling that $d_n = 3^n$.

We construct a quasiregular map $f$ depending on the sequence $(\lambda_n)$ as follows.
\begin{itemize}
\item On $\overline{B(0,R_1)}$ we define $f(x) = p_3(x)$.
\item On $A(R_1,S_1)$ we define $f(x) = F_{R_1,3,1}(x)$.
\item On $A(S_1,R_2)$, we set $C_1 = 1$ and define $f(x) = C_1 \cdot p_9(x)$.
\item For $n\geq 1$, suppose $C_n$ is given, then on $A(R_{n+1},S_{n+1})$ we define $f(x) = F_{R_{n+1},d_{n+1},C_{n}}(x)$.
\item Then we set $C_{n+1} = C_n R_{n+1} ^{ -2 d_{n+1}}$ and on $A(S_{n+1},R_{n+2})$ we define $f(x) = C_{n+1} \cdot p_{d_{n+2}}(x)$.
\end{itemize}

Thus on rings of the form $A(R_n,S_n)$ we interpolate, whereas on rings of the form $A(S_n,R_{n+1})$ we hold up the growth of the map by keeping it as a power-type map.

For $r\ge 0$, we denote the maximum modulus by $M(r,f)= \max\limits_{|x|=r}|f(x)|$.

\begin{lemma}
\label{lem:formulas}
For $n\geq 0$, we have
\[ R_{n+1} = \left ( \prod_{j=1}^n \lambda_j \right ) \exp \left ( \sum_{j=1}^n \frac{ 1}{d_j} \right ), \quad \text{ and } \quad
S_{n+1} = \left ( \prod_{j=1}^n \lambda_j \right ) \exp \left ( \sum_{j=1}^{n+1} \frac{ 1}{d_j} \right ).\]
Moreover, we have 
\[ M(R_{n+1},f) = \left ( \prod_{j=1}^n \lambda_j^{d_{j+1}} \right ) e^{3n}, \quad \text{ and } \quad M(S_{n+1},f) = \left ( \prod_{j=1}^n \lambda_j^{d_{j+1}} \right ) e^{3(n+1)}.\]
\end{lemma}

\begin{proof}
This is an elementary induction calculation that we omit.
\end{proof}

Observe that to obtain a quasiregular mapping on $\R^3$ via this construction, we need
\begin{equation}\label{eq:prod-lambda_j}
    \prod_{j=1}^n \lambda_j \to \infty \quad \mbox{ as } n\to \infty.
\end{equation}

\begin{theorem}
Consider the constructions above.
\begin{itemize}
\item Suppose that  $(\eta_k)$ is any real sequence. Then we may choose $(\lambda_n)$ satisfying \eqref{eq:prod-lambda_j} so that there exists a sequence $r_k\to\infty$ with $r_k <2^k$ and 
\begin{equation}\label{eq:biggrowth}
    \frac{\log \log M(r_k,f) }{ \log r_k} \geq \eta_k.
\end{equation} 
\item Given any positive sequence $(\epsilon_k)$, we may choose $(\lambda_n)$ satisfying \eqref{eq:prod-lambda_j} such that, for $k\ge 2$,
\begin{equation}\label{eq:slowgrowth}
\frac{ \log \log M(R_k,f) }{\log R_k} \leq  \epsilon_k.
\end{equation}
\end{itemize}
\end{theorem}

In other words, the first part says that the transcendental type quasiregular map $f$  can have arbitrarily fast growth, and the second part says that $f$ may have arbitrarily slow growth.

\begin{proof}
By Lemma \ref{lem:formulas}, we have
\begin{align*}
\frac{ \log \log M(R_{n+1},f) }{\log R_{n+1}} &= \frac{ \log \log \left [  \left ( \prod_{j=1}^n \lambda_j^{d_{j+1}} \right ) e^{3n}  \right ] }{ \log \left [  \left ( \prod_{j=1}^n \lambda_j \right ) \exp \left ( \sum_{j=1}^n \frac{ 1}{d_j} \right ) \right ] } \\
&= \frac{ \log \left [  \sum_{j=1}^n 3^{j+1} \log \lambda_j + 3n  \right ] }{ \sum_{j=1}^n \log \lambda_j + \sum_{j=1}^n \frac{1}{3^j}  }.
\end{align*}

For the first part, take $\lambda_1,\ldots,\lambda_{n_1}=1$, where $n_1$ is large. Then, by Lemma~\ref{lem:formulas}, $R_{n_1+1}\le e^{1/2}$ and $M(R_{n_1+1}, f) = e^{3n_1}$. Take $r_1=R_{n_1+1}$. Then $r_1\le e^{1/2}<2$ and \eqref{eq:biggrowth} holds for $k=1$ provided $n_1$ is chosen large enough.

We now inductively define a sequence $(\lambda_n)$ and a subsequence $(r_k)$ of $(R_n)$ as follows. Suppose that for some $k\ge 1$ we have chosen integers $n_1, \ldots, n_k$ and values $\lambda_1,\ldots,\lambda_{n_k}$ such that $r_k=R_{n_k+1}$ satisfies $r_k<2^k$ and \eqref{eq:biggrowth}. We then set $\lambda_{n_k+1} = 2e^{-1/2}$ and $\lambda_j=1$ for $j=n_k+2, \ldots, N$ where $N$ is chosen large enough that
\[ \frac{ \log \log M(R_{N+1} , f) }{\log R_{N+1}} = \frac{ \log \left [  \sum_{j=1}^{n_k+1} 3^{j+1} \log \lambda_j + 3N  \right ] }{ \sum_{j=1}^{n_k+1} \log \lambda_j + \sum_{j=1}^N \frac{1}{3^j}  } > \eta_{k+1}.\]
(Note that the denominator here is bounded as $N\to\infty$ but the numerator is unbounded.) We set $n_{k+1}=N$ and $r_{k+1}=R_{N+1}$ so that \eqref{eq:biggrowth} holds for $k+1$. Moreover,
\[ r_{k+1} = 2e^{-1/2}\exp\left(\sum_{j=n_k+1}^N \frac{1}{d_j}\right)r_k \le 2r_k  < 2^{k+1}.\]
This completes the induction and the sequence $(r_k)$ has the required properties. Since $\lambda_j$ equals $2e^{-1/2}>1$ infinitely often, we see that \eqref{eq:prod-lambda_j} holds and $r_k\to\infty$.

The second part is also proved by induction. Suppose we are given $(\epsilon_k)$ and that $\lambda_1,\ldots,\lambda_{n-1}$ have been chosen such that $R_k$ satisfies \eqref{eq:slowgrowth} for $k=2,\ldots,n$ (we make no assumption in the case that $n=1$). Writing $y=\log \lambda_n$, we have that
\[ \frac{ \log \log M(R_{n+1},f) }{\log R_{n+1}} =  \frac{ \log \left [  \sum_{j=1}^{n-1} 3^{j+1} \log \lambda_j + 3n  +3^{n+1}y \right ] }{ \sum_{j=1}^{n-1} \log \lambda_j + \sum_{j=1}^{n} \frac{1}{3^j} + y } \to 0\]
as $y\to\infty$. Thus choosing $\lambda_n$ large enough gives \eqref{eq:slowgrowth} for $k=n+1$.
\end{proof}

\begin{rem}
    If we follow the above construction but with $\lambda_n=1$ for all $n$, then we obtain a quasiregular map $f\colon B(0,e^{1/2})\to\R^3$. This map coincides with a scaled power map $C_n\cdot p_{d_{n+1}}$ on each sphere $\{|x|=R_{n+1}=\exp(\frac12(1-3^{-n}))\}$, and $|f(x)|=e^{3n}$ on these spheres.
\end{rem}

\end{document}